\documentclass[twocolumn]{autart}    

\usepackage{cite}
\usepackage{graphicx}      
\usepackage{psfrag}
\usepackage{amsfonts}
\usepackage{amssymb}
\usepackage{amsmath}
\usepackage{enumerate}
\usepackage{appendix}
\usepackage{psfrag,epsfig,graphicx,ae}
\usepackage{color}

\usepackage{verbatim}

\newcounter{teocount}
\newcounter{propcount}
\newcounter{corcount}
\newcounter{remcount}
\newcounter{defcount}

\newtheorem{remm}[remcount]{Remark}

\newtheorem{definition}[defcount]{Definition}
\newtheorem{proposition}[propcount]{Proposition}
\newtheorem{theorem}[teocount]{Theorem}
\newtheorem{corollary}[corcount]{Corollary}

\newtheorem{exx}{Example}

\newenvironment{remark}{\begin{remm}\rm }{\hfill \hspace*{1pt} \hfill $\lrcorner$\end{remm}}

\newenvironment{proof}{{\em Proof. }}{\hfill \hspace*{1pt}
\hfill $\blacksquare$}

\newcommand\real{\ensuremath{{\mathbb R}}}

\newcommand{\R}{{\mathbb R}}

\newcommand{\cA}{{\mathcal A}}
\newcommand{\cB}{{\mathcal B}}
\newcommand{\cC}{{\mathcal C}}
\newcommand{\cD}{{\mathcal D}}
\newcommand{\cE}{{\mathcal E}}

\newcommand{\cI}{{\mathcal I}}

\newcommand{\cM}{{\mathcal M}}

\newcommand{\cS}{{\mathcal S}}

\newcommand{\cX}{{\mathcal X}}

\begin{document}
\begin{frontmatter}


\title{Region of Attraction Estimation Using Invariant Sets and Rational Lyapunov Functions\thanksref{footnoteinfo}} 



\thanks[footnoteinfo]{ At the time of writing G. Valmorbida was  affiliated to the Department of Engineering Science and  Somerville College, University of Oxford, Oxford, U.K. J. Anderson is funded via a Junior Research Fellowship from St. John's College, University of Oxford, Oxford, U.K.  Work partly supported by EPSRC grant EP/J010537/1.}

\author[Gif]{Giorgio Valmorbida}\ead{giorgio.valmorbida@l2s.centralesupelec.fr},               
\author[Oxford]{James Anderson}\ead{james.anderson@eng.ox.ac.uk},  
\address[Gif]{Laboratoire des Signaux et Syst\`emes, CentraleSup\'elec, CNRS, Univ. Paris-Sud, Universit\'e Paris-Saclay, 3~Rue~Joliot-Curie, Gif sur Yvette 91192, France. \\[-.6cm]} 
\address[Oxford]{Department of Engineering Science,  University of Oxford, Parks Road, Oxford, OX1 3PJ, UK.\\[-.6cm]}             


\begin{abstract}                
This work addresses the problem of estimating the region of attraction (RA) of equilibrium points of nonlinear dynamical systems. The estimates we provide are given by positively invariant sets which are not necessarily defined by level sets of a Lyapunov function. Moreover, we present conditions for the existence of Lyapunov functions linked to the positively invariant set formulation we propose. Connections to fundamental results on estimates of the RA are presented and support the search of Lyapunov functions of a rational nature. We then restrict our attention to systems governed by polynomial vector fields and  provide an algorithm that is guaranteed to enlarge the estimate of the RA at each iteration. 
\end{abstract}


\end{frontmatter}

\section{Introduction}\label{sec:intro}

The problem of computing the region of attraction~(RA) of asymptotically stable equilibria, or inner estimates to this set (ERA)~\cite{CHW88}, is central in several applications and its relevance is immediately clear for many practical nonlinear systems for which we can only guarantee local properties of operating points.

With a converse Lyapunov theorem \cite[Theorem 19]{Zub64}, Zubov answered the question ``\textit{[...] Is it possible, with the help of the Lyapunov function  to find a region of variation of the initial values $x_0$ such that $\|\phi(t, x_0)\|\rightarrow 0 \quad \mbox{as} \quad t \rightarrow \infty$ ?}'' \cite[p.3]{Zub64}. The theorem states that if $\mathcal{S}$ \emph{is} the RA of an equilibrium then the existence of a Lyapunov function (LF) satisfying some conditions on such a set $\mathcal{S}$ is \textit{necessary and sufficient}. However, computing the  LF  and the exact RA following Zubov's theorem requires the solution of a partial differential equation, which is difficult to obtain in all but simple cases. 
However, \textit{local} solutions (in a compact set around the equilibrium point) to the conditions can be obtained more easily and yield ERAs for the equilibrium point of interest. In this context, a method to approximate solutions to the conditions of \cite[Theorem 19]{Zub64} is obtained with a series expansion of the LF  \cite[p.91]{Zub64} and is now referred to as \textit{Zubov's Method}.



In \cite{VV85}, Zubov's theorem was modified to consider Lyapunov functions mapping $\R^n$ to $\R_{\geq 0}$ (the original result is stated in terms of a map from $\R^n$ to the interval $[-1,0]$).  One of the conditions in \cite{VV85} imposes that the LF $V(x)$ satisfies $V(x) \rightarrow \infty$ whenever $x \rightarrow \partial \mathcal{S}$ (the boundary of the RA) or whenever $\|x\|~\rightarrow~\infty$. Such a property is described by the observation that \textit{``the candidate must in effect `blow up' near the boundary of the domain of attraction''.} These functions were called \textit{maximal Lyapunov functions} (MLFs).  One of the key observations was that rational functions could be used to approximate MLFs and therefore be used to obtain estimates of the RA.  As a matter of fact, the class of rational functions of the form $V(x) = \frac{V_N(x)}{V_D(x)}$
where $V_N$ and $V_D$ are polynomials, were considered as LF candidates in the algorithm proposed in \cite{VV85} with the boundary of the ERA characterised by the set $\{x\in \R^n \mid V_D(x) = 0\}$.

At this point, for the sake of clarity, it is important to distinguish between two similar sounding yet very different objects: a maximal Lyapunov function (MLF) and a maximal Lyapunov set (MLS). An MLF is a Lyapunov function which satisfies a strict set of conditions (cf. Definition~\ref{def:MLF} in Section~\ref{sec:main}). In contrast, a MLS is defined as the largest level set of a \emph{given} LF contained in a specified set. Computing the MLS is of course of interest since one might wish to compute the best ERA achievable for a given Lyapunov function~\cite{Che04a,Che13}. Further to the choice of the class of the LF, conservativeness is introduced by imposing the level sets of the Lyapunov function to be the ERA, as observed in \cite[p.320]{Kha02}  \textit{``Estimating the region of attraction by $\Omega_c = \{x|V(x) \leq c\}$ is simple but usually conservative. According to LaSalles's theorem [...] we can work with any compact set $\Omega \subset D$ provided we can show that $\Omega$ is positively invariant.''}  The statement highlights the fact that contractiveness of the function defining the ERA is restrictive.

In recent years, sufficient conditions for local stability analysis, requiring invariance and contractiveness of a set led to numerical methods for the estimation of the RA with polynomial Lyapunov functions~\cite{TA08,TP09,TPSB10}.
These methods rely on the solution of non-convex sum-of-squares (SOS) constraints constructed with the \textit{Positivstellensatz} \cite[Theorem 2.14]{Las09}. The solutions to these problems require a coordinate-wise search since the non-convex nature results from the fact that some polynomial variables appear multiplying the Lyapunov function which is itself a variable.  For a detailed description of sum-of-squares methods for RA estimation the reader is referred to \cite{Che11}. For the case of a given LF, the computation of the MLS was pursued in \cite{Che13}. In \cite{HenL12} the theory of moments is used to estimate the RA of uncertain polynomial systems.  We also find in the literature numerical methods exploiting  topological properties of the boundary of the RA requiring the computation of trajectories and equilibrium points. However the complexity of such methods has restricted them to 2-dimensional examples~\cite{CHW88}. Recently, in \cite{WanLC11}, set advection methods are described for polynomial systems.

 In this paper we derive conditions based on Lyapunov stability results that guarantee that trajectories initiated from an \emph{positively invariant set} converge to a level set of the LF which is contractive and invariant therefore guaranteeing such a postively invariant set to be an ERA. In addition to the positively invariant estimates, we present conditions to obtain LF certificates of a specific form which specializes to rational functions in case of polynomial data. We then propose a numerical method based on the solution of SOS constraints for the case of polynomial systems and estimates in the form of semi-algebraic sets (sets defined by polynomial constraints).  The work in this paper extends the work of \cite{ValA14} and connects the concept of \textit{maximal Lyapunov functions} \cite{VV85} to polynomial optimization techniques based on sum-of-squares programming. To the best of the authors knowledge this is the first work to offer a theoretical link between maximal Lyapunov functions, which completely characterise the ERA (and can be approximated to arbitrary accuracy by rational functions) and sum-of-squares methods for rational LF construction. Note that rational Lyapunov functions were considered in \cite{Che13} to obtain MLSs. 
 
The paper is organised as follows: We  present some definitions and the problem statement in Section \ref{sec:prelim} and describe the main theoretical results in Section \ref{sec:main}. Narrowing our attention to systems described by polynomial vector fields we describe a computational method for constructing ERAs based on sum-of-squares programming in Section \ref{sec:ERAviaLyap}  which is illustrated by numerical examples in Section \ref{sec:examples}.

\section{Preliminaries}
\label{sec:prelim}

Let $\R, \R_{\ge 0}, \R_{>0}$ and $\R^{n}$ denote the field of reals, non-negative reals, positive reals and the $n$-dimensional Euclidean space respectively. The function $f:\R^{n}\rightarrow \R$ is positive definite if $f(x)>0$ for all non-zero $x\in \R^{n}$, similarly if $f(x)\ge 0$ for all $x\in \R^{n}$ then $f$ is positive semidefinite. The set of functions $g:\R^{n}\rightarrow \R$ which is  $n$-times continuously differentiable is denoted~$\cC^{n}$. $co(\cX)$ denotes the convex hull of the set $\cX$, $\cX^\circ$ its interior, $\partial \cX$ its boundary,  and $\overline{\cX}$ its closure. The minimum (maximum) of a scalar function $S(x)$ in a compact set $\mathcal{Y}$ is denoted $\displaystyle{\min_{x\in \mathcal{Y}}(S(x))}$ ($\displaystyle{\max_{x\in \mathcal{ Y}}(S(x))}$). We also use $\max$ to denote the function taking the maximum of its arguments. For $x\in \R^{m}$ the ring of polynomials in $m$ variables is denoted by $\R[x]$. For $p\in\R[x]$, $deg(p)$ denotes the degree of $p$.  A polynomial  $p(x)$ is said to be a sum-of-squares if there exists a finite set of polynomials $g_{1}(x),\hdots,g_{k}(x)$ such that $p(x)=\sum_{i=1}^{k}g_{i}^{2}(x)$. The set of SOS  polynomials in $x$ is denoted by $\Sigma[x_{1},\hdots,x_{m}]$ which can be abbreviated to $\Sigma[x]$. Equivalently $p(x)$ is SOS if there exists a positive semidefinite matrix $Q$ such that $p(x)=Z^{T}(x)QZ(x)$ where $Z(x)$ is a  vector of monomials  \cite{Par00}. Note that the search for $Q$ can be formulated as a semidefinite programme and thus solved using convex optimization techniques~\cite{VanB96}. 

Consider the dynamical system 
\begin{equation}\label{eq:sys}
\dot{x}=f(x)
\end{equation} 
where $f:\cD\rightarrow \R^{n}$ is a locally Lipschitz map from a domain $\cD\subset \R^{n}$ to $\R^{n}$, with $0\in \cD$. Let us assume $x=0$ is an equilibrium point, \textit{i.e.} $0 \in \{x\in \R^n | f(x) = 0\}$. Denote by $\phi(t,x(0))$ the solution to (\ref{eq:sys}) that is initiated from the point $x(0)$ at time $t=0$, the set $L$ is said to be \emph{invariant} with respect to (\ref{eq:sys}) provided $x(0) = \phi(0,x(0))\in L\Rightarrow x(t) = \phi(t,x(0))\in L$, $\forall t\in \R$. Furthermore, $L$ is said to be \emph{positively invariant} with respect to (\ref{eq:sys}) if the previous implication holds for all $t\ge 0$. Given a function $R:\R^{n}\rightarrow \R$ we define the set $\cE(R,\gamma):= \left\{x \in \R^{n}\text{ }\left|\text{ } R(x)\le \right. \gamma \right\}$ for some $\gamma>0$ and $\cE^{\circ}(R,\gamma):= \left\{x \in \R^{n}\text{ }\left|\text{ } R(x)< \right. \gamma \right\}$ for some $\gamma>0$. Additionally, provided that a function $V:\R^{n}\setminus \{0\}\rightarrow \R_{> 0 }$ satisfies $\dot{V}(x)= \frac{\partial V}{\partial x}f(x)< 0$  on $\cE(V,\gamma)$ then the set $\cE(V,\gamma)$ is said to be \emph{contractive} and \emph{invariant}, furthermore the function $V$ is said to be a Lyapunov function \cite[Chapter 4]{Kha02}. 
 We assume throughout this work that any function used to define a contractive set is in $\cC^{1}$. The region of attraction of an asymptotically stable equilibrium point $x^{*}$ of (\ref{eq:sys}) is defined as the set
\begin{equation}\label{eq:S}
\mathcal{S}	 := \left\{x\in \R^n \text{ }\left| \begin{array}{c} \text{ }\phi(t,x) \text{ is defined }\forall t\ge 0,\\ \text{ }\lim_{t \rightarrow \infty}\phi(t,x) = x^{*} \end{array}\right. \right\},
\end{equation}
without loss of generality, we will assume throughout this paper that the equilibrium point of interest is at the origin, i.e. $x^*=0$.

The focus of this paper is to construct inner estimates of $\mathcal{S}$ by computing positively invariant sets.



\section{Main results}
\label{sec:main}
In this section we present conditions to certify that a compact set is a postively invariant set and provides an estimate of the RA for the origin of~\eqref{eq:sys}. It is also shown how, under a different condition, to obtain an alternative Lyapunov certificate. We then extend these results to the case where the system under study is affected by parametric uncertainty.
\subsection{Region of Attraction Estimates}
The following theorem is used to verify that a compact set is positively invariant and defines an estimate of the RA of the equilibrium point at the origin and to obtain functions of which the denominator provides the RA estimate.

\begin{theorem}
\label{thm:inv}
Given $R: \R^n \rightarrow \R$, $R \in \cC^1$ and $\gamma >0$, satisfying
\begin{subequations}\label{eq:R}
\begin{align}
\cE(R,\gamma) \subset \cD~\mbox{is compact  and}~~0\in \cE(R,\gamma),\label{eq:compactness} \\
-\left\langle \nabla R(x),  f(x) \right \rangle  > 0  \quad  \forall x \in  \partial \cE(R,\gamma) , \label{eq:Rineq}
\end{align}
\end{subequations}
if there exists  $V_N: \R^n~\rightarrow~\R$, $V_N(0) = 0$, $V_N \in \cC^1$, such that
\begin{subequations}
\label{eq:thmstab}
\begin{equation}
\label{eq:Vposineq}
V_N(x) >0 \ \forall x \in  \cE(R,\gamma) \setminus \{0\},
\end{equation}
\begin{equation}
-\left\langle \nabla V_N(x), f(x)\right \rangle  >0 \ \forall x \in  \cE(R,\gamma) \setminus \left\{0 \right\},
\label{eq:Vdotineq}
\end{equation}
\end{subequations}
then
\begin{enumerate}[(I)]
\item the solutions $x(t) = \phi(t, x_0)$ to~\eqref{eq:sys},  with $x_0 \in  \mathcal{E}(R, \gamma)$ satisfy $x(t')\in \cS ~\forall t' \in [0,\infty).$
\end{enumerate}
Moreover, if \eqref{eq:compactness}, \eqref{eq:Vposineq}  and 
\begin{multline}
-\left\langle \left[ \nabla V_N(x)(\gamma - R(x)) + V_N(x)\nabla R(x) \right], f(x)\right \rangle  >0 \\ \forall x \in    \cE(R,\gamma)  \setminus \{0\}
\label{eq:Vratdotineq}
\end{multline}
hold and there exists a positive scalar $\bar{\epsilon}$ such that $0 \notin \cE(R,\gamma ) \setminus \cE^\circ(R,\gamma-\bar{\epsilon}) $ then
\begin{enumerate}[(II)]
\item the function
\begin{equation}
\label{eq:LyapRational}
V(x) = \dfrac{V_N(x)}{\gamma - R(x)}
\end{equation}
\end{enumerate}
is a Lyapunov function for~\eqref{eq:sys} and gives $\mathcal{E}^{\circ}(R, \gamma)$ as an estimate of $\cS$.
\end{theorem}
\begin{proof}
Proof of (I). \emph{Part 1} (Contractiveness of a level set of $V_N$): by assumption $\cE(R,\gamma)$ is compact thus we can compute $\alpha = \displaystyle{\min_{x \in \partial \cE(R,\gamma)}} V_N(x)$, then from \eqref{eq:Vposineq}, $\alpha>0$. 
Define $\mathcal{U} = \cE(V_N,\alpha) \cap \cE(R,\gamma)$, from \eqref{eq:Vposineq}, \eqref{eq:Vdotineq} we have
\label{eq:propstabproof}
\begin{equation*}
V_N(x) >0 \ \forall x \in \mathcal{U} \setminus \left\{0 \right\},
\end{equation*}
\begin{equation*}
-\left\langle \nabla V_N(x), f(x)\right \rangle  >0 \ \forall x \in \mathcal{U} \setminus \left\{0 \right\}.
\end{equation*}
Following \cite[Theorem 4.1]{Kha02} we have that the origin is asymptotically stable and an inner approximation of its region of attraction is given by $\mathcal{U}$, that is trajectories $\phi(t, x_0)$ with $x_0 \in \mathcal{U}$ exist, are unique, and satisfy $\phi(t, x_0)\rightarrow 0$ as $t\rightarrow \infty$.

\emph{Part 2} (Positively Invariance of $\cE(R,\gamma)$): since $\cE(R,\gamma)$ is compact and $f(x)$ is locally Lipschitz in any compact set, we have existence and uniqueness of solutions to $\dot{x} = f(x)$, for all $x_0  \in \cE(R,\gamma)$, provided every solution lies in $\cE(R,\gamma)$. Let us prove that for all $T \in [0,\infty)$ we have $x(T)\in \cE(R,\gamma)$. Assume there exists $x_0  \in \cE(R,\gamma)$ for which the solution leaves the set, then, there must exist a $T^{*}$ such that $x(T^{*}) = \phi(T^{*}, x_0)$ that satisfies $R(x(T^{*}))> \gamma$. From the continuity of solutions and continuity of $R(x)$ there must exist $\bar{T}$, $0<\bar{T}<T^{*}$ such that $R(x(\bar{T})) = \gamma$ and $\dot{R}(x(\bar{T})) = \langle \nabla R(x(\bar{T})), f(x(\bar{T})) \rangle \ge 0$, which contradicts \eqref{eq:Rineq}. Hence $\cE(R,\gamma)$ is a   positively  invariant set. 

\emph{Part 3} (Convergence of trajectories starting in $  \cE(R,\gamma)\setminus  \cE^{\circ}(V_N,\alpha)$ to $\mathcal{U}$):  Finally let us now prove that every trajectory satisfying $x(0) \in  ~\cE(R,\gamma) \setminus \cE^{\circ}(V_N, \alpha)$ enters the positively  invariant and contractive set $\mathcal{U}$, that is, that there exists a $T$ such that $x(T) \in \mathcal{U}$. Let $\beta = \displaystyle{\max_{x \in \partial \cE(R,\gamma)} V_N(x)}$. Since  \eqref{eq:Vdotineq} holds in $\cE(R,\gamma) \setminus \cE^{\circ}(V_N, \alpha)$ let $\lambda = -\displaystyle{\max_{x\in\cE(R,\gamma) \setminus \cE^{\circ}(V_N, \alpha)} \langle \nabla V_N,f(x) \rangle}$, which exists because the continuous function has a maximum over any compact set, from \eqref{eq:Vdotineq} we also get
$$V_N(x(t)) = V_N(x(0)) + \int_0^t \dot{V}_N(x(\tau))d\tau \leq V_N(x(0)) - \lambda t.$$
Since $V_N(x(0))\leq \beta$ we have $V_N(x(t)) \leq \beta - \lambda t$ in the set $\cE(R,\gamma) \setminus \cE^{\circ}(V_N, \alpha)$. This implies that $\exists T\geq0$ satisfying $T\leq \frac{\beta-\alpha}{\lambda}$ such that $V_N(x(T)) = \alpha$ and hence $x(T)\in \mathcal{U}$.

Proof of (II):  From \eqref{eq:Vposineq} we have that $V(x)>0~\forall x \in \cE^{\circ}(R, \gamma)\setminus \{0\}$.  The time-derivative of \eqref{eq:LyapRational} along the trajectories of \eqref{eq:sys} is given by 
\begin{equation}
\label{eq:VratdotFULL}
\dot{V}(x) = \dfrac{ \left\langle  \nabla V_N(x) (\gamma - R(x)) + V_N(x) \nabla{R} , f(x) \right\rangle  }{(\gamma - R(x))^2}
\end{equation}
which satisfies $-\dot{V}(x) >0~\forall x \in \cE^{\circ}(R, \gamma)\setminus \{0\}$ if \eqref{eq:Vratdotineq} holds true. Thus $V$ is a Lyapunov function for the equilibrium point at the origin of \eqref{eq:sys}. What is left to show is that the level curves of $V$ define the stated estimate of the ERA.


We have that $V(x)$ and $\dot{V}(x)$ are not defined in $\partial \cE(R, \gamma)$, thus it is not possible to compute scalars $\alpha$ and $\beta$ Parts~1 and~3 in the above Proof of (I). Consider  the arbitrarily small scalar, $\epsilon$ that defines the set $ \cE(R, \gamma - \epsilon)$ and assume it satisfies $\epsilon \leq \bar{\epsilon}. $ Then, as $\cE(R, \gamma)$ is compact it follows that so is $ \cE(R, \gamma - \epsilon)$. We can then follow the same steps of Proof of (I), Part 1 and compute a positive scalar $\alpha_\epsilon = \displaystyle{\min_{x \in \partial \cE(R,\gamma- \epsilon)}} V(x)$. Define $\mathcal{U}_\epsilon = \cE(V,\alpha_\epsilon) \cap \cE(R,\gamma-\epsilon)$, since $V(x)>0~\forall x \in \cE^{\circ}(R, \gamma)\setminus \{0\}$ and $-\dot{V}(x)>0~\forall x \in \cE^{\circ}(R, \gamma)\setminus \{0\}$ we have
\begin{align*}
V(x) >0 \quad \forall x \in \mathcal{U}_\epsilon \setminus \left\{0 \right\},\\
-\left\langle \nabla V(x), f(x)\right \rangle  >0 \quad \forall x \in \mathcal{U}_\epsilon \setminus \left\{0 \right\}.
\end{align*}
Then following \cite[Theorem 4.1]{Kha02} the origin is asymptotically stable and an inner approximation of its region of attraction is given by $\mathcal{U}_\epsilon$, that is trajectories $\phi(t, x_0)$ with $x_0 \in  \mathcal{U}_\epsilon$ exist, are unique, and satisfy $\phi(t, x_0)\rightarrow 0$ as $t\rightarrow \infty$.

Using \eqref{eq:LyapRational}, we obtain  
\begin{equation*}
\cE(V,\alpha_\epsilon) = \left\lbrace  x \in \real^{n} \mid R(x) \leq \gamma - \frac{V_N(x)}{\alpha_\epsilon} \right\rbrace.
\end{equation*}
Since 
\begin{equation*}
\begin{array}{rcl}
\alpha_\epsilon & = & \displaystyle{\min_{x \in \partial \cE(R, \gamma-\epsilon)}} V(x)\\
& = & \displaystyle{\min_{x \in \partial \cE(R, \gamma-\epsilon)}} \frac{V_N(x)}{\gamma-R(x)}\\
& = & \displaystyle{\min_{x \in \partial \cE(R, \gamma-\epsilon)}} \frac{V_N(x)}{\gamma-(\gamma - \epsilon)}\\
& = & \dfrac{1}{\epsilon}\left(\displaystyle{\min_{x \in \partial \cE(R, \gamma-\epsilon)}}  V_N(x)\right)\\
\end{array}
\end{equation*}
we then have
\begin{equation}
\cE(V,\alpha_\epsilon) = \left\lbrace  x \in \real^{n} \mid R(x) \leq \gamma - \epsilon \frac{V_N(x)}{ \displaystyle{\min_{x \in \partial \cE(R, \gamma-\epsilon)}}  V_N(x) } \right\rbrace.
\label{eq:setIdentity}
\end{equation}
Define a positive scalar  $\delta$ as follows: 
\begin{equation*}
\delta:=  \frac{ \displaystyle \max_{x \in  \cE(R, \gamma)} V_N(x)}{ \displaystyle \min_{x \in  \cE(R, \gamma) \setminus \cE^{\circ}(R,\gamma-\bar{\epsilon})} V_N(x) }
\end{equation*}
then it is immediate that $\delta>1$.  As $\epsilon\leq \bar{\epsilon}$ we have  $\partial \cE(R, \gamma-\epsilon) \subset  \cE(R, \gamma) \setminus \cE^{\circ}(R,\gamma-\bar{\epsilon})$, thus $\displaystyle{\min_{x \in \partial \cE(R, \gamma-\epsilon)} V_N(x) \geq \min_{x \in \cE(R, \gamma) \setminus \cE^{\circ}(R,\gamma-\bar{\epsilon})} V_N(x)}$  and hence $$\delta >\dfrac{V_N(x)}{\displaystyle{\min_{x \in \partial \cE(R, \gamma-\epsilon)}}  V_N(x)},~\forall x \in \cE(R, \gamma).$$

Since $\epsilon$ is positive, the set containment
\begin{equation*} \cE(R, \gamma-\epsilon\delta) \subset \left\lbrace  x \in \real^{n} \mid R(x) \leq \gamma - \epsilon \frac{V_N(x)}{ \displaystyle{\min_{x \in \partial \cE(R, \gamma-\epsilon)}}  V_N(x) } \right\rbrace
\end{equation*}
holds true. Using \eqref{eq:setIdentity}  we have $\cE(R, \gamma-\epsilon\delta) \subset  \cE(V, \alpha_\epsilon)$, thus $\cE(R, \gamma-\epsilon\delta) = \cE(R, \gamma-\epsilon\delta) \cap \cE(R, \gamma-\epsilon) \subset \cE(V, \alpha_\epsilon) \cap \cE(R, \gamma-\epsilon)  = \mathcal{U}_\epsilon$ which, using the fact that $\mathcal{U}_\epsilon \subset \cE(R, \gamma)$, we obtain
\begin{equation*} \cE(R, \gamma-\epsilon\delta) \subset  \mathcal{U}_\epsilon \subset \cE(R, \gamma).
\end{equation*}
Thus $\cE(R, \gamma-\epsilon\delta)$ is an estimate of the RA of~\eqref{eq:sys}. Since~$\delta$ is bounded and $\epsilon$ can be chosen to be arbitrarily small, we have that the set $\cE^{\circ}(R, \gamma)$ is an estimate of the RA of~\eqref{eq:sys}.\end{proof}

A consequence of the the assumption imposed in  Theorem \ref{thm:inv} that the set $\mathcal{E}(R,\gamma)$ is compact, is that $\mathcal{E}(R,\gamma)$ will be connected. This follows from the fact that ~\eqref{eq:Rineq} makes $\mathcal{E}(R,\gamma)$  a positively invariant set, and by~\eqref{eq:Vposineq},  $V_N(x)$ is strictly positive on $\mathcal{E}(R,\gamma) \setminus \{0\}$, and finally ~\eqref{eq:Vdotineq} ensures its derivative is strictly negative on $\mathcal{E}(R,\gamma) \setminus \{0\}$.

\begin{remark}
Note that  $R(x)$ is not required to be positive definite, however it is required that $\cE(R,\gamma)$ is compact and contains the  origin. This requirement guarantees that, in the proof of Part 1, $\min_{\partial \cE(R,\gamma)}(V(x))$ is well-defined, such that $\cE(V,\alpha) \subseteq \cE(R,\gamma)$. 
\end{remark}


\begin{remark}
The boundedness of the set $\cE(R,\gamma)$ also guarantees the uniqueness of solutions in the set if the vector field is not globally Lipschitz (as for instance, the polynomial vector fields). 
\end{remark}

\begin{remark}
In the case $R(x)=V_N(x)$, \eqref{eq:Vdotineq} implies~\eqref{eq:Rineq} and the set $\cE(V_N,\gamma)$  has to be compact as required by~$\eqref{eq:compactness}$.
\end{remark}
The proposition below presents sufficient conditions to satisfy the constraints of Theorem~\ref{thm:inv} formulated in terms of inequalities and the definition of  $ \cE(R, \gamma)$. 
\begin{proposition}\label{prop:ineq}
Given $R \in \mathcal{C}^1$, $R: \R^n \rightarrow \R$, $\gamma>0$, satisfying
\begin{equation}
\cE(R,\gamma)~\mbox{is compact},~~0 \in \cE(R,\gamma),
\label{eq:compactnessSet}
\end{equation}
if there exist $V_N: \R^n \rightarrow \R$, $V_N \in \mathcal{C}^{1}$, $V_N(0) = 0$ and $m_0 : \R^n \rightarrow \R_{\geq 0 } $, $m_1 : \R^n \rightarrow \R_{\geq 0 }$, $p : \R^n \rightarrow \R$,  such that
\begin{subequations}
\label{eq:propstab}
\begin{eqnarray}
-\left\langle \nabla R(x),  f(x) \right \rangle  > p(x)(\gamma - R(x) ) &&  \forall x \in \cD \label{eq:RineqSet} \\
V_N(x) > m_0(x)(\gamma - R(x)) && \forall x \in \cD'  \label{eq:VposineqSet} \\
-\left\langle \nabla V_N(x), f(x)\right \rangle  > m_1(x)(\gamma - R(x)) && \forall x \in \cD'  \label{eq:VdotineqSet} 
\end{eqnarray}
\end{subequations}
\noindent
where $\cD' := \cD\setminus \{ 0\}$, then~$\mathcal{E}(R, \gamma)$ is an ERA of the origin. If~\eqref{eq:compactnessSet},~\eqref{eq:VposineqSet},~\eqref{eq:RineqSet} hold and there exist $m_2 : \R^n \rightarrow \R_{\geq 0 } $ such that 
\begin{equation}
-\left\langle \nabla V_N(x) , f(x)\right \rangle + V_N(x)p(x) > m_2(x)(\gamma - R(x))
\label{eq:VratdotineqSet}
\end{equation}
holds, then~\eqref{eq:LyapRational} is a Lyapunov function for~\eqref{eq:sys} and $\cE^{\circ}(R,\gamma)$ is an ERA of the origin.
\end{proposition}
\begin{proof}
From the non-negativity  of $m_0(x)$, and $m_1(x)$ we have that \eqref{eq:VposineqSet}, \eqref{eq:VdotineqSet} imply that \eqref{eq:Vposineq}, \eqref{eq:Vdotineq} hold. Since at $\partial \cE(R, \gamma)$ we have $ \gamma - R(x)   = 0$, \eqref{eq:RineqSet} implies \eqref{eq:Rineq}, and, according to Theorem~\ref{thm:inv} Claim~(I), $\cE(R, \gamma)$ is an ERA of the origin. The time derivative of $V(x)$, as in~\eqref{eq:LyapRational}, is given by~\eqref{eq:VratdotFULL}. Since we have $V_N(x)>0$ in $\cE(R, \gamma)$ if \eqref{eq:VposineqSet} holds, inequality~\eqref{eq:RineqSet} provides a lower bound for $-\dot{V}$ as follows
\begin{equation*}
\begin{array}{rcl}
- \dot{V}(x) & = & -\dfrac{\left\langle \nabla V_N(x) ,f(x)\right \rangle}{(\gamma -R)  } -\dfrac{V_N(x) \left\langle \nabla R(x), f(x)\right \rangle}{(\gamma -R)^2 } \\
&>  & -\dfrac{\left\langle \nabla V_N(x) ,f(x)\right \rangle}{(\gamma -R)  } +\dfrac{V_N(x) p(x)}{(\gamma -R) } ,
\end{array}
\end{equation*}
therefore, if there exists $m_2  : \R^n \rightarrow \R_{>0}$ satisfying
\begin{equation*}
 \dfrac{ \left( - \left\langle \nabla V_N(x), f(x)\right\rangle  + V_N(x)p(x) \right)}{(\gamma -R)} > m_2(x),
\end{equation*}
that is \eqref{eq:VratdotineqSet},
then $-\dot{V}(x)>0$ (and \eqref{eq:Vratdotineq} holds true). Following Theorem~\ref{thm:inv}, Claim (II), $V(x)$ is a Lyapunov function for~\eqref{eq:sys} in $\cE^{\circ}(R, \gamma)$,  thus providing an ERA for the origin of~\eqref{eq:sys}.
\end{proof}

Before developing the theory further in order to take into account non-smooth set descriptions and ERAs for uncertain systems we first compare the above conditions to the classical results of Zubov via Vannelli and Vidyasagar's \emph {maximal Lyapunov function} \cite{VV85} framework which characterises the region of attraction $\cS$.
\begin{definition}
\label{def:MLF}
A function $V_m:\R^n \rightarrow \R_{>0}  \cup \left\{ \infty \right\}$ that for the system \eqref{eq:sys} satisfies
\begin{enumerate}
\item $V_m(0)=0, V_m(x)>0$ if $x\in \cS\setminus \left\{ 0 \right\}$,
\item $V_m(x)<\infty$ iff $x\in \cS$,
\item $V_m(x)\rightarrow \infty$ as $x\rightarrow \partial \cS$ and/or $\|x\|\rightarrow \infty$,
\item $\dot{V}_m(x)<0$ and well defined for all $x\in \cS \setminus  \left\{ 0 \right\}$,
\end{enumerate}
is called a maximal Lyapunov function.
\end{definition}
The main result from \cite{VV85} is summarised by the following theorem:
\begin{theorem}\label{thm:VV}
Suppose we can find a set $\cA\subseteq \R^n$ which contains $x=0$ in its interior, a continuously differentiable function $V_m:\cA\rightarrow \R_{>0}$ and a positive definite function $\psi(x)$ such that
\begin{enumerate}
\item  $V_m(0)=0, V(x)>0$ if $x\in \cA\setminus \left\{ 0 \right\}$,
\item  $\dot{V}_m(x) = - \psi(x)$  for all $x\in \cA$,
\item $V_m(x)\rightarrow \infty$ as $x\rightarrow \partial \cA$ and/or $\|x\|\rightarrow \infty$,
\end{enumerate}
then $\cA=\cS$.
\end{theorem}
Clearly an MLF satisfies the properties of Theorem~\ref{thm:VV}, however in~\cite{VV85} the authors constructively provide a method for extending any LF for~\eqref{eq:sys} into an MLF (the assumption is that $\psi(x)$  has been constructed). They then show that rational LFs of the form $V(x) = \frac{V_N(x)}{V_D(x)}$ with polynomial numerator and denominator can arbitrarily approximate an MLF. Furthermore the  set $\{x\mid V_D(x)=0\}$ defines the boundary of $\mathcal{S}$.

In comparison to Theorem~\ref{thm:VV}, Theorem~\ref{thm:inv} shows that the zero level set of a rational Lyapunov function can be used to compute an estimate of the RA. The main difference between the two results is that our theorem provides an estimate of the RA whilst the MLF provides the exact RA (albeit at the cost of having to solve a partial differential equation, namely item~2 of Theorem~\ref{thm:VV}). Thus the LF \eqref{eq:LyapRational} is not necessarily an MLF. In the remainder of this section we develop further the results of Theorem \ref{thm:inv}.

\subsection{Piece-wise Lyapunov Functions}
The following result parallels Theorem~\ref{thm:inv} and considers positively invariant regions defined by the maximum of a set of differentiable functions. 

Let $d$ be a finite positive integer and  functions $R_i(x),\hdots, R_d(x)$ be given. The \emph{point-wise maximum} of these functions at $x$ is defined as
\begin{equation}
\label{eq:Rmax}
R_M(x)  := \max (R_1(x), \ldots,R_d(x)).
\end{equation}
\begin{theorem}
\label{thm:invPW}
Given  $R_i: \R^n \rightarrow \R$, $R_i\in \cC^{1}$ $i = 1, \ldots, d$ and a positive scalar~$\gamma$, satisfying
\begin{subequations}
\label{eq:prop2a}
\begin{equation}
\label{eq:compactness2}
\cE(R_M,\gamma)~\mbox{is compact  and}~~0 \in \cE(R,\gamma),
\end{equation}
\begin{equation}
-\left\langle \xi,  f(x) \right \rangle > 0 \\ \ \forall x \in  \partial \cE(R_M,\gamma), \forall\xi \in \dfrac{\partial R_M(x)}{\partial x}
\label{eq:Rineqratcomp}
\end{equation}
\end{subequations}
where $\frac{\partial R_M(x)}{\partial x}$ denotes the subdifferential of $R_M(x)$ at $x$, if there exists a function $V_N: \R^n~\rightarrow~\R_{\geq 0}$, $V_N(0) = 0$, $V_N \in \cC^{1}$  such that
\begin{subequations}
\label{eq:prop2}
\begin{equation}
V_N(x) >0 \ \forall x \in  \cE(R_M,\gamma) \setminus \{0\}
\label{eq:Vposineqratcomp}
\end{equation}
\begin{equation}
- \left\langle \nabla V_N(x), f(x) \right \rangle >0 \ \forall x \in \cE(R_M,\gamma) \setminus \{0\}
\label{eq:Vdotineqratcomp}
\end{equation}
\end{subequations}
then
\begin{enumerate}[(I)]
\item all trajectories of~\eqref{eq:sys} initiated from the set $\cE(R_M,\gamma)$ converge to the origin. 
\end{enumerate}
Moreover, if  \eqref{eq:Vposineqratcomp}, and 
\begin{multline}
-\left\langle \left[ \nabla V_N(x)(\gamma - R_M(x)) + V_N(x)\nabla R_M(x) \right], f(x)\right \rangle  >0 \\ \forall x \in    \cE(R_M,\gamma) 
\label{eq:Vratdotineqcomp}
\end{multline}
hold, and there exists a positive scalar $\bar{\epsilon}$ such that $0 \notin \cE(R_M,\gamma ) \setminus \cE^\circ(R_M,\gamma-\bar{\epsilon}) $ then 
\begin{enumerate}[(II)]
\item the function
\begin{equation}
\label{eq:LyapRationalcomp}
V(x) = \dfrac{V_N(x)}{\gamma - R_M(x)}
\end{equation}
is a Lyapunov function for~\eqref{eq:sys} and gives $\mathcal{E}^{\circ}(R_M, \gamma)$ as an estimate of $\cS$.
\end{enumerate}
\end{theorem}

\begin{remark}
To characterise the set $\cE(R_M,\gamma)$, notice that $R_M(x) = R_i(x) \ \forall x \in \{x \in\R^n | R_i(x)-R_j(x)\geq 0, j = 1, \ldots, d \}$. By defining
\begin{multline*}
\cM_i(R_M, \gamma) := \{x \in\R^n |  R_i(x) \leq \gamma, \\ R_i(x)-R_j(x)\geq 0, j = 1, \ldots, d \}.
\end{multline*}
we can write
$\cE(R_M,\gamma) = \bigcap_{i =1}^d \cM_i(R_M, \gamma)$.
\end{remark}

The subdifferential for the function $R_M(x)$ is defined as  $\frac{\partial R_M(x)}{\partial x} := co\{\nabla R_{\ell}(x), \ell \in \cI(x)\}, $ where $\cI(x) = \{i \in \{1, \ldots, d \} | R_i(x) = R_M(x)  \}$ denotes the set of ``active'' functions at point $x$. Notice that $R_M(x)$ is not differentiable at points $x$ where $\exists i,j \in \cI(x), i~\neq~j$, that is, at points satisfying $R_M(x)=R_i(x)=R_j(x), i \neq j$. At such points $\frac{\partial R_M(x)}{\partial x}$ defines a set, hence \eqref{eq:Rineqratcomp} describe a set of inequalities. Whenever $\cI(x)$ contains only one element, say $\cI(x) = \{k\}$,  $\frac{\partial R_M(x)}{\partial x}$ is a singleton given by $\nabla R_{k}$, which exists since $R_i(x)\in \cC^1, \forall i = 1, \ldots, d$.

The proof of Theorem \ref{thm:invPW} follows closely the proof  of Theorem \ref{thm:inv} and is therefore omitted. The only difference is related to the lack of differentiability of $R_M(x)$ which gives $\dot{R}_M(x(t)) \in \left\langle \frac{\partial R_M(x(t))}{\partial x} , f(x(t))\right\rangle$. Therefore provided \eqref{eq:Rineqratcomp} holds we can use it to arrive at a contradiction as in the proof of Claim (I) of Theorem \ref{thm:inv}.

The proposition below parallels Proposition~\ref{prop:ineq} and is presented without proof. It introduces  sufficient conditions to satisfy the conditions of Theorem~\ref{thm:invPW}. These conditions are formulated in terms of inequalities and the description of the set $\partial \cE(R_M, \gamma)$. 
\begin{proposition}\label{prop:ineqPW}
Given  $R_i: \R^n \rightarrow \R$, $R_i\in \cC^{1}$ $i = 1, \ldots, d$ and a positive scalar~$\gamma$, if there exist $V_N: \R^n \rightarrow \R$, $V_N \in \mathcal{C}^{1}$, $V_N(0) = 0$, and $m_0 : \R^n \rightarrow \R_{\geq 0 } $, $m_1 : \R^n \rightarrow \R_{\geq 0 } $, $p : \R^n \rightarrow \R$ such that 
\begin{subequations}
\label{eq:propstabPW}
\begin{eqnarray}
\cE(R_M,\gamma)~\mbox{is compact},~~0 \in \cE(R_M,\gamma)\label{eq:compactnessSetPW} \\
-\left\langle \xi,  f(x) \right \rangle > p(x)(\gamma - R_M(x) ) \quad \forall \xi \in \frac{\partial R_M(x)}{\partial x} \label{eq:RineqSetPW} \\
V_N(x) > m_0(x)(\gamma - R_M(x)) \label{eq:VposineqSetPW}\\
-\left\langle \nabla V_N(x), f(x)\right \rangle  > m_1(x)(\gamma - R_M(x)) \label{eq:VdotineqSetPW}
\end{eqnarray}
\end{subequations}

\noindent
then~$\mathcal{E}(R_M, \gamma)$ is an ERA of the origin. If~\eqref{eq:compactnessSetPW},~\eqref{eq:VposineqSetPW}, \eqref{eq:RineqSetPW} hold and there exist $m_2 : \R^n \rightarrow \R_{\geq 0 } $ such that 
\begin{equation}
-\left\langle \nabla V_N(x) , f(x)\right \rangle + V_N(x)p(x) > m_2(x)(\gamma - R_M(x))
\label{eq:VratdotineqSetPW}
\end{equation}
hold, then~\eqref{eq:LyapRationalcomp} is a LF for~\eqref{eq:sys} in the set $\mathcal{E}^{\circ}(R_M, \gamma)$.
\end{proposition}

\subsection{Uncertain systems}

Consider uncertain dynamical systems of the form
\begin{equation}\label{eq:uncertain}
\dot{x} = f(x,\theta),\quad \theta \in \Theta \subset \R^{n_p}
\end{equation}
where $f: \cD \times \Theta \rightarrow \R^{n}$ and $\Theta$ denotes the uncertainty set. We assume that $f$ satisfies conditions so as to provide uniqueness and local existence of solutions\footnote{This requires $f$ to be continuous in $(x,\theta,t)$ and locally Lipschitz in $x$ (uniformly in $\theta$ and $t$) on a bounded domain. Exact conditions can be found in \cite[Theorem 3.5]{Kha02}}. We shall also assume that $x^{*}=0$ is the equilibrium of interest, and require that $f(0,\theta)=0$  $\forall \theta \in \Theta$. 

We are interested in determining a robust estimate for region of attraction \textit{i.e.} an estimate of the RAs for all dynamical systems of the form \eqref{eq:uncertain} with a fixed $\theta\in\Theta$
\begin{equation*} 
\cS_{\theta} := \left\{  x_{0}\in \R^n \text{ }\left| \begin{array}{c} \text{ }\phi(t,x,\theta) \text{ is defined }\forall t \ge0 \\ 
\lim_{t \rightarrow \infty}\phi(t,x_{0},\theta) = x^{*}, \forall \theta \in \Theta \end{array}\right. \right\},
\end{equation*}
where $\phi(t,x_{0},\theta)$ is a solution to~\eqref{eq:uncertain} starting from $x_{0}$ at time $t$ with fixed $\theta\in \Theta$. We establish conditions for positively invariant sets to be estimates of the region of attraction for parametrically uncertain systems. Whilst we will consider parameter dependent Lyapunov functions (PDLFs), our ERAs will be defined by positively invariant sets which are not dependent on the system parameters. 

PDLFs have been shown to be an effective tool for certifying the stability of linear system with parametric uncertainties \cite{GahAC96,OP07,Bli04}. They have also been successfully applied to obtain certificates for the local stability of polynomial systems leading to parameter-dependent estimates of the RA. In those results, a robust estimate is then obtained as the intersection of the estimates given for each fixed parameter value~\cite{Che04a,TPSB10}, which is in contrast to the result in this section where the estimate does not depend on the parameters but the Lyapunov function does, thus avoiding the computation of the intersection of the parametrised estimates.


The following result extends Theorem \ref{thm:inv} to the case of uncertain systems of the form~\eqref{eq:uncertain}.
\begin{theorem}
\label{thm:invUnc}
Consider the uncertain dynamical system described by \eqref{eq:uncertain} where $\Theta$ is a compact set and  $x^*=0$. Given a function $R: \R^n \rightarrow \R$, $R \in \cC^1$ and  a positive scalar $\gamma$, which defines a compact set $\mathcal{E}(R, \gamma)$, and satisfy 
\begin{equation}
-\left\langle \nabla R(x),  f(x,\theta) \right \rangle  > 0  \ \forall (x,\theta) \in  \partial \cE(R,\gamma) \times \Theta,   \label{eq:RineqUnc} 
\end{equation}
if there exists a function $V_N: \R^n \times \Theta~\rightarrow~\R$, $V_N(0, \cdot) = 0$, $V_N \in \cC^1$   such that
\begin{equation}
V_N(x, \theta) >0 \ \forall  (x,\theta) \in  \cE(R,\gamma) \setminus \{0\} \times \Theta \label{eq:VposineqUnc}
\end{equation}
\begin{equation*}
-\left\langle \nabla V_N(x,\theta), f(x,\theta)\right \rangle  >0 \ \forall (x,\theta) \in  \cE(R,\gamma) \setminus \{0\} \times \Theta  \nonumber
\end{equation*}
\noindent
then  the solutions to \eqref{eq:uncertain}, $\phi(t,x_{0},\theta)$ for any  $x_{0}\in\cE(R,\gamma)$ and $\theta \in \Theta$ lie in the set $\cS_\theta$ with respect to  $x^*=0$.

Moreover, if \eqref{eq:VposineqUnc}, and 
\begin{multline}
-\left\langle \left[ \nabla V_N(x, \theta)(\gamma - R(x)) + V_N(x,\theta)\nabla R(x) \right], f(x)\right \rangle  >0 \\ \forall (x,\theta) \in  \partial \cE(R,\gamma) \times \Theta
\label{eq:VratdotineqUnc}
\end{multline}
hold, and there exists a positive scalar $\bar{\epsilon}$ such that $0 \notin \cE(R,\gamma ) \setminus \cE^\circ(R,\gamma-\bar{\epsilon}) $ then the function
\begin{equation}
\label{eq:LyapRationalUnc}
V(x,\theta) = \dfrac{V_N(x,\theta)}{\gamma - R(x)}
\end{equation}
is a Lyapunov function for~\eqref{eq:uncertain} for all $\theta \in \Theta$ in the set $\mathcal{E}^{\circ}(R, \gamma)$.
\end{theorem}
The proof is similar to that of Theorem \ref{thm:inv} and so is omitted. 

\section{Computational Methods for Estimating the~RA}
\label{sec:ERAviaLyap}
We now present a computational method for constructing positively invariant estimates of the RA. First  a method for estimating the RA via maximal Lyapunov sets is reported and then algorithms that implement the main results of the paper are described.
\subsection{Maximal Lyapunov Sets}
Recall that for a locally asymptotically stable equilibrium point of~(\ref{eq:sys}), converse Lyapunov theorems tells us there  exist a Lyapunov function $V$ and a set $\cD\subset \R^{n}$, ${0} \in \cD$, satisfying $V:\cD\rightarrow \R$ such that $V(x)>0$ $\forall x\in \cD\setminus \left\{ 0\right\}$, $V(0)=0$ and $\left\langle \nabla V(x), f(x)\right\rangle<0$ $\forall  x\in \cD\setminus \left\{ 0\right\}$, see for example \cite[Section 4.7]{Kha02}. Based on this fact, a common (but conservative) approach to finding an ERA for systems of the form~(\ref{eq:sys}) is to compute a Lyapunov function certifying the local asymptotic stability of the origin and obtain the largest (maximal) level set of a Lyapunov function that is contained within the set~$\cD$  which the LF is constructed on. We can describe a general algorithm based on the Lyapunov function computations to obtain ERAs in the form $\cE(V,\gamma)$  as follows:\\ \\
\textbf{\underline{Algorithm 1}}\\
\textbf{Input} $k = 0$, a compact set $\cD_0$, $\{0\}\in \cD_0$.\\
\textbf{Step 1 }\textit{(Lyapunov function computation):} Given $\cD_k$, compute a Lyapunov function $V_k$ for (\ref{eq:sys}).\\
\textbf{Step 2 }\textit{(Maximization of Lyapunov level set):} Given $V_k$ and $\cD_k$, solve 
\begin{equation}
\text{maximize}\  \gamma \quad \text{subject to} \ \cE(V_k,\gamma)\subset \cD_k, \ \gamma >0
\label{eq:alg1opt}
\end{equation}
\textbf{Step 3 }\textit{(Update of search domain):} If stopping criteria is satisfied then return $\cE(V_k, \gamma^*)$, with $\gamma^*$ the solution to~\eqref{eq:alg1opt}, as the ERA else specify ~$\cD_{k+1}$, set $k \leftarrow (k+1)$ go to Step 1. \hfill$\blacksquare$

Whilst the algorithm above may look simple enough, observe that: i) in general, constructing Lyapunov functions for nonlinear dynamical systems is a non-trivial task. ii) Existing formulations for the optimization problem in Step~2 are typically non-convex. iii) Determining the update for $\cD_{k+1}$ typically relies on some heuristic. As a general rule, by necessity $\left\{ 0 \right\} \in \cD_{k+1}$, the set should be connected, and should contain points that are not already in $\cD_{k}$. In Section \ref{sec:poly} we will specialise the above algorithm and explicitly describe how to update the search domain, and in the case of polynomial systems construct all the required functions. 


Step 1 asks for $\left\langle \nabla V_k(x), f(x)\right\rangle<0 \ \forall x \in \cD_k$.  In general there is no guarantee that every $x(0) \in \cD_k$ satisfies $\phi(t, x(0)) \in \cD_k \ \forall t>0$. Extra conditions must hold for $\cD_k$ to be a positively invariant set. We will expand upon this point in the next section.
 
Denote the ERA and the search domain from iteration~$k$ of the above algorithm by $\cE(V_{k},\gamma_{k})$ and $\cD_k$ respectively. Note that $\cD_k\subset \cD_{k+1}$ does not necessarily guarantee $\cE(V_{k},\gamma_{k}) \subset \cE(V_{k+1},\gamma_{k+1})$. Satisfying such constraints  is of desirable as it guarantees improvement of the ERA.  For polynomial systems, we describe next how such criteria can be enforced. 

\subsection{Estimating the RA With Invariant Sets}\label{sec:ERAinvariant}
We now illustrate how Theorem~\ref{thm:inv}, Claim (I) can be implemented in an algorithmic manner to obtain estimates for the RA. We start by  presenting a generic algorithm, analogous to Algorithm~1 with the exception that the obtained ERA is not given by Lyapunov level sets:\\ \\
\textbf{\underline{Algorithm 2}}\\
\textbf{Input}: $k=0$,  a function $R_0$ satisfying the~\eqref{eq:compactness} in Theorem~\ref{thm:inv}.\\
\textbf{Step 1 }(\textit{Invariant set enlargement}): Maximize $\gamma$ subject to \eqref{eq:Rineq}, \eqref{eq:thmstab} with $R = R_k$.\\
\textbf{Step 2 }(\textit{Update function $R$}): If stopping criteria is satisfied then return ERA given by the set $\cE(R_k,\gamma)$; else compute  $R^*$ satisfying $\cE(R^*,\delta)\supseteq \cE(R_k,\gamma^*)$ , where~$\gamma^*$ is the optimal solution to Step 1 and $\delta>0$. Set $k\leftarrow k+1$;  $\gamma \leftarrow \delta$ ; $R_k \leftarrow R^*$; go to Step 1.\hfill$\blacksquare$


Note that any function $R_0$ that satisfies \eqref{eq:R} can be used to initiate Algorithm 2 provided the set $\cE(R_0,\gamma)$ is invariant. A straightforward choice for $R_0$ is any Lyapunov function $V$, which provides a (possibly arbitrarily small) level set  $\cE(V,\gamma)$ which is \emph{invariant} and also contractive (although such a property is not required for $\cE(R_0,\gamma)$). We take such an approach  in the examples  in Section~\ref{sec:examples}.

Observe  that the update of function $R$ parallels the update of the domain in Algorithm 1, with the difference that it defines an ERA itself. In the next subsection we describe a specific strategy for the update $R^*$ of function $R$ from step~2. 

\subsection{Polynomial systems}\label{sec:poly}
For the remainder of the paper it is assumed that the vector field $f$ in \eqref{eq:sys} and the functions $V_N$, $R$ and $R_M$ in Theorems~\ref{thm:inv} and \ref{thm:invPW} are multivariate polynomials. For this class of systems,~\cite{Che11} presents a comprehensive set of results on estimates of the domain of attraction with polynomial Lyapunov functions with Sum-of-Squares based approaches~\cite{BlePT13}.

The following corollary to Proposition \ref{prop:ineq} provides sufficient conditions for Theorem \ref{thm:inv} to hold which are verifiable using convex optimization.

\begin{corollary}
Let $V_N$ and $R$ be given multivariable polynomials and $\gamma$ a given positive constant. Then, if there exist sum-of-squares polynomials $m_0,m_1$ and a polynomial $p$ such that 
\begin{subequations}\label{eq:psatz_thm1}
\begin{eqnarray}
-\left\langle \nabla R(x),f(x) \right\rangle -p(x)(\gamma - R(x))\in \Sigma[x] \label{eq:psatz_thm1d} \\
V_N(x)-m_{0}(x)(\gamma - R(x))\in \Sigma[x] \label{eq:psatz_thm1a}  \\
-\left\langle \nabla V_N(x),f(x) \right\rangle -m_{1}(x)(\gamma - R(x))\in \Sigma[x] \label{eq:psatz_thm1b} 
\end{eqnarray}
\end{subequations}
then the inequalities  \eqref{eq:propstab}  are satisfied and~$\mathcal{E}(R, \gamma)$ is an ERA of the origin. If \eqref{eq:psatz_thm1d},~\eqref{eq:psatz_thm1a} hold and there exist a sum-of-squares polynomial $m_2 $ such that 
\begin{equation}
-\left\langle \nabla V_N(x),   f(x) \right\rangle +  V_N(x)p(x) -m_{2}(x)(\gamma - R(x))\in \Sigma[x]\label{eq:psatz_thm1c}
\end{equation}
hold, then~\eqref{eq:LyapRational} is a LF for~\eqref{eq:sys} in the set $\mathcal{E}^{\circ}(R, \gamma)$.
\end{corollary}

As proven in Proposition~\ref{prop:ineq}, the fact that \eqref{eq:psatz_thm1d} and \eqref{eq:psatz_thm1c} hold is a sufficient condition for~\eqref{eq:Vratdotineq} and, if satisfied, it certifies that the rational function {$V = \frac{V_N}{\gamma-R}$} is a Lyapunov function on the set $\cE^{\circ}(R,\gamma)$. 

Sum-of-squares constraints such as those above can be formulated using freely available software such as SOSTOOLS~\cite{sostools} and solved using a semidefinite programme solver. Note that in~\eqref{eq:psatz_thm1},  $V_N$ and $R$ appear affinely in the SOS constraints, and the only product between $V_N$ and the set of multipliers is in~\eqref{eq:psatz_thm1c} where it multiplies polynomial $p$. This fact is central in developing coordinate-wise search algorithms allowing  $V_N$ to be a variable at all steps of the algorithm iterations. Notice that is only possible because in~\eqref{eq:psatz_thm1} there is no product between $V_N$ and the multipliers $m_0$, $m_1$ and $m_2$. When searching for a rational LF certificate (by imposing~\eqref{eq:psatz_thm1c}), the product $V_N p$ is handled by fixing $p$ from a solution to~\eqref{eq:psatz_thm1d}.

\begin{remark}
Notice that, for a given rational $V = \frac{V_N}{\gamma-R}$, (assuming polynomial dependence of $V_N$ and $R$ on $x$) straightforward formulations to compute the MLS of the function $V$, as the set $\cE(V,C)$, lead to inequalities as
\begin{multline*}
-\dfrac{\left\langle \nabla V_N(x)(\gamma - R(x)) + V_N(x)\nabla R(x), f(x) \right\rangle}{(\gamma - R)^2} \\ \geq m(x)\left( C - \dfrac{V_N(x)}{\gamma - R(x)}\right).
\end{multline*}
with $m(x)\geq0$. In order to avoid the rational inequalities of the form above, one can restrict the attention to the set where $(\gamma - R(x))>0$, thus formulating polynomial inequalities as
\begin{multline}
\label{eq:highdegree}
- \left\langle \nabla V_N(x)(\gamma - R(x)) + V_N(x)\nabla R(x), f(x) \right\rangle  \\ \geq m(x)(C(\gamma - R(x))^2 - V_N(x)(\gamma - R(x)) ).
\end{multline}
Notice also that, in contrast to \eqref{eq:highdegree}, \eqref{eq:psatz_thm1c} does not present products $V_N(x)R(x)$ and $R(x)^2$. Such a property also allows the degree of the SOS constraints in \eqref{eq:psatz_thm1c}  to be lower than a SOS constraint obtained with~\eqref{eq:highdegree}. 
\end{remark}

\begin{remark}
\label{rem:RequalsV}
The analysis of polynomial systems with  polynomial LFs and level sets of the LF as ERAs are a particular case of the conditions  imposed  and are obtained by imposing $R(x) = V_N(x)$, $m_0(x) = 0$, and $p(x) =m_1(x) = m_2(x)$.
\end{remark}

The proof of the following claims are straightforward and, therefore, omitted. 

\begin{corollary}
If \eqref{eq:psatz_thm1d} holds with $p\in \R[x]$, that also satisfies $p\in \Sigma[x]$ (i.e. a sum-of-squares polynomial), then \eqref{eq:psatz_thm1c} is holds true with $m_2=m_1+pm_0$.
\end{corollary}

\begin{prop}
Given polynomials $\widehat{R}$, $R$ and a scalar $\gamma>~0$, if there exists a sum-of-squares polynomial $m_3(x)$ such that
\begin{equation}
\label{eq:RinS}
(\gamma -R(x))-m_{3}(x)(\gamma-\widehat{R}(x))\in \Sigma[x],  m_{3}\in \Sigma[x].
\end{equation}
then $\cE(\widehat{R},\gamma)\subseteq \cE(R,\gamma)$.
\end{prop}

The constraints \eqref{eq:psatz_thm1} and \eqref{eq:RinS} can be used to formulate the following bilinear sum-of-squares programme 
\begin{align*}
\underset{V_N,R,m_{i},p}{\text{maximize}} \quad &\gamma   \\
\text{subject to}\quad &\text{\eqref{eq:psatz_thm1}, \eqref{eq:RinS}, $m_i \in \Sigma[x], i\in \left\{0,1\right\}$.}
\end{align*}
From~\eqref{eq:RinS}, any solution $\gamma^{*}$ to the above problem guarantees that the set $\cE(R,\gamma^{*})$ is contained in the set $\cE(\widehat{R},\gamma^{*})$ which is an ERA since  conditions in Theorem~\ref{thm:inv} are satisfied when~\eqref{eq:psatz_thm1} is satisfied.

 The algorithm below exploits this fact to specialize Algorithm~2 and obtain ERAs for polynomial systems by solving a sequence of sum-of-squares programmes. \\ \\
\textbf{\underline{Algorithm 3}}\\
\textbf{Input: }$k=0$,  a function $R_0$ satisfying the condition~\eqref{eq:compactness} of~(I) in Theorem~\ref{thm:inv}.\\
\textbf{Step 1: }With $R=R_k$, solve through a line search on~$\gamma$:
\begin{equation*}
\underset{V_N,m_{0},m_{1},p}{\text{maximize}} \quad \gamma \quad \text{subject to \eqref{eq:psatz_thm1a}-\eqref{eq:psatz_thm1d}}.
\end{equation*}
\textbf{Step 2: }If stopping criterion is satisfied, return $\cE(R_k,\gamma)$ else, using $m_{0}, m_{1}$,  $p$ and the optimal $\gamma^{*}$ from Step~1, set $\widehat{R}(x)=R(x)$ and solve with a line search on $\gamma$:
\begin{equation}
\underset{V_N,R, m_{3}}{\text{maximize}} \ \gamma  \text{ subject to \eqref{eq:psatz_thm1}, \eqref{eq:RinS}},  \gamma \geq \gamma^{*},
\label{prob:step2}
\end{equation}
Set $k\leftarrow k+1$; $R_k \leftarrow R^*$ with $R^{*}$ the optimizer of~\eqref{prob:step2}; go to Step 1.   \hfill$\blacksquare$

Notice that the Lyapunov function $V_N$ is a decision variable in every step of the above algorithms. This is not the case if one imposes $R = V_N$ as mentioned in Remark~\ref{rem:RequalsV}. 

Algorithm 3 guarantees a sequence of non-decreasing ERAs and a function $R(x)$ satisfying \eqref{eq:psatz_thm1} is required for its initialisation. This function can be taken as any Lyapunov function satisfying the local stability of the origin, for example a quadratic function of the form $V(x)=x^TPx$ for some positive definite $P$ if the linearized system matrix $A$ is Hurwitz. The estimate obtained from running Algorithm 3 depends on the initial function $R_0$, hence running Algorithm 3 with different initialisations $R_0$ may lead to better estimates, one choice for $R_0$ is the Lyapunov function $V_N$ produced by the algorithm (see Example 2 in Section \ref{sec:examples}).

After using Algorithm 3 to construct an ERA, it is desirable to compute the corresponding rational LF of the form \eqref{eq:LyapRational}. In order to do so, using the multiplier $p$, functions $R$ and $\gamma$ from the final iteration of Algorithm~3 solve the feasibility problem 
\begin{equation*}
\underset{V_N,m_{0},m_{1},m_{2}}{\text{find}} \quad V_N \quad \text{subject to \eqref{eq:psatz_thm1}, \eqref{eq:psatz_thm1c}, $m_i \in \Sigma[x]$,}
\end{equation*}
$i\in \left\lbrace0,1,2 \right\rbrace$. One important assumption in Theorem~\ref{thm:inv} is the compactness of set $\cE(R,\gamma)$. In order to enforce this property when computing $R(x)\in\R[x]$ we impose a constraint of the form
$R(x) \geq c(x) + \kappa \|x\|^{2k}$
where $c(x)\in\R[x]$, $deg(c) \leq 2k-1$, $\kappa \in \real_{\geq 0}$, $k \in \mathbb{N}$.

\section{Numerical Examples} \label{sec:examples}

We now illustrate our results with three numerical examples, we use SOSTOOLS and the semidefinite programme solver SeDuMi~\cite{Stu99}.

\textbf{Example 1: }
Consider system~\eqref{eq:sys}  with 
\begin{equation*}
f(x) = \left[\begin{array}{c}
-0.42 x_1 - 1.05x_2 -2.3x_1^2 - 0.5x_1x_2 - x_1^3\\
1.98x_1 + x_1x_2
\end{array}\right]
\label{eq:traj}
\end{equation*}
which satisfies $f(0) = 0$. This system was studied in~\cite[Example 4]{VV85} in the context of maximal LFs and in~\cite{TA08} with composite Lyapunov functions. It describes a Lotka-Volterra system with its stable equilibrium point translated to the origin.  

With the initial function $R$ obtained from the ERA from~\cite{VV85}, we apply Algorithm~3 (thus allowing the Lyapunov function $V_N$ to be a variable at each step).  With the obtained positively invariant set $\cE(R,\gamma)$ defining the ERA and a multiplier $p$ solving~\eqref{eq:psatz_thm1d}, we compute $V_N$ satisfying \eqref{eq:psatz_thm1a}, \eqref{eq:psatz_thm1b}, \eqref{eq:psatz_thm1c} and $m_i \in \Sigma[x]$, $i \in \{0,1,2\}$ thus yielding a rational LF as~\eqref{eq:LyapRational}. The boundary of the ERA and a sequence of nested level sets of the rational LF $V(x)$ are depicted in Figure \ref{fig:MLFlevelsets}. Figure~\ref{fig:MLFlevelVRdot} depicts the sets $\{ x |\dot{V}_N = 0\}$ and the set $\{ x |\dot{R} = 0\}$,  illustrating that constraints \eqref{eq:psatz_thm1b} and \eqref{eq:psatz_thm1d} hold, that is, the intersection of $\partial\cE(R,\gamma)$ with $\{ x |\dot{R} \geq 0\}$ is empty.  For comparison purposes, Figure~\ref{fig:MLFlevelVndot} depicts the ERA,  the maximal level set obtained with the numerator $V_N$, which is strictly contained in the ERA and the set $\{ x |\dot{V}_N = 0\}$. This feature illustrates the conservativeness of the estimate obtained by computing the MLS of a given polynomial LF.  
\begin{figure}[!htb]
\begin{psfrags}
     \psfrag{x2}[r][r][1][-90]{\footnotesize $x_2$}
     \psfrag{x1}[l][l]{\footnotesize $x_1$}
\epsfxsize=7.5cm
\centerline{\epsffile{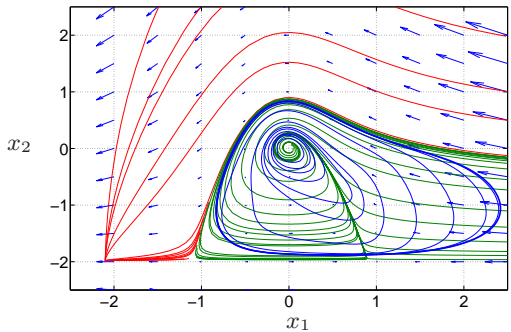}}
\end{psfrags}
\caption{Dark blue line dark depicts the boundary of the ERA, i.e. the set $\partial \cE(R,\gamma)$. Level sets of the function $V(x)$ are also depicted. Trajectories depicted in green converge to the origin.
\label{fig:MLFlevelsets}}
\end{figure}
\begin{figure}[!htb]
\begin{psfrags}
     \psfrag{x2}[r][r][1][-90]{\footnotesize $x_2$}
     \psfrag{x1}[l][l]{\footnotesize $x_1$}
\epsfxsize=7.5cm
\centerline{\epsffile{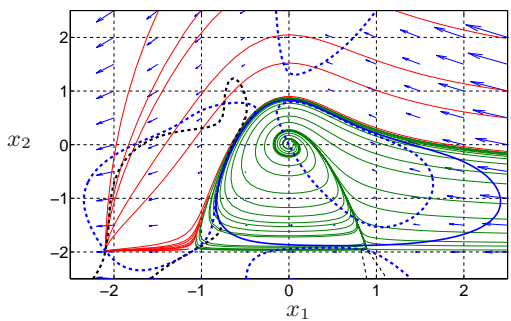}}
\end{psfrags}
\caption{The dashed light blue lines depict the set of points satisfying $\dot{R}(x) = \langle \nabla R(x), f(x)\rangle = 0$,  the dashed black lines depict the set $\dot{V}(x) =\langle \nabla V(x), f(x)\rangle = 0$. The boundary of the RA lies in the set where $\dot{R}(x)$ is negative.
\label{fig:MLFlevelVRdot}}
\end{figure}
\begin{figure}[!htb]
\begin{psfrags}
     \psfrag{x2}[r][r][1][-90]{\footnotesize $x_2$}
     \psfrag{x1}[l][l]{\footnotesize $x_1$}
\epsfxsize=7.5cm
\centerline{\epsffile{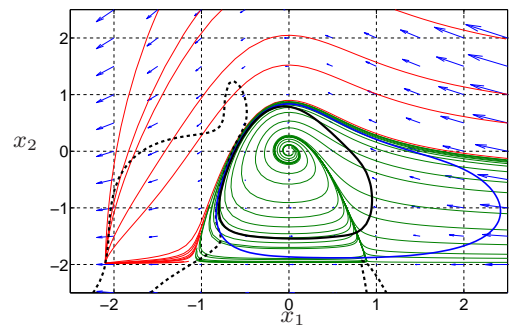}}
\end{psfrags}
\caption{The boundary of the ERA,  $\partial\cE(R,\gamma)$, is depicted in dark blue. The dashed black lines depict the set $\dot{V}_N(x) = 0$. The solid black line depicts the MLS obtained with $V_N$ as the LF. 
\label{fig:MLFlevelVndot}}
\end{figure}

\textbf{Example 2: }The following three-dimensional system from~\cite[Example 5]{VV85} presents a limit cycle and an stable equilibrium at the origin:
\begin{equation*}
f(x) = \left[\begin{array}{c}
-x_2\\
-x_3 \\
-0.915x_1 + (1 - 0.915x_1^2)x_2-x_3 
\end{array}\right].
\label{eq:traj2}
\end{equation*}
We apply Algorithm~3 starting with $R$ given by a quadratic LF for the linearised system and set $deg(V_N) = 4$. We then use the obtained LF, $V_N$, as the initial invariant set function $R$ and apply Algorithm~3 again. The obtained ERA of degree two and degree four are depicted in Figure~\ref{fig:maxRi3d}. We were unable to find  multipliers $p$ and $m_2$ that satisfy \eqref{eq:psatz_thm1c} for the computed $V_N$ and $R$, hence we could not construct a rational LF certificate of the form~\eqref{eq:LyapRational}.


\begin{figure}[!htb]
\begin{psfrags}
     \psfrag{x3}[r][r][1][-90]{\footnotesize $x_3$}
     \psfrag{x1}[l][l]{\footnotesize $x_1$}
     \psfrag{x2}[l][l]{\footnotesize $x_2$}     
\epsfxsize=8cm
\centerline{\epsffile{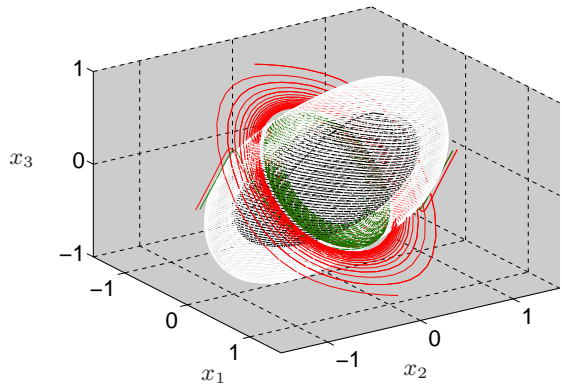}}
\end{psfrags}
\caption{The boundary of the ERA, $\partial \cE(R,\gamma)$, of degree $deg(R) = 4$ corresponds to the white surface while the black surface corresponds to the largest ERA obtained with $deg(R) =2$. Trajectories  are depicted in green (converging) and red (diverging). 
\label{fig:maxRi3d}}
\end{figure}

\textbf{Example 3: }In the following example we compute a piece-wise positively invariant set. Consider now system~\eqref{eq:sys}  with  
\begin{equation}
f(x) = \left[\begin{array}{c}
-x_1(1 -x_1x_2)\\
-x_2
\end{array}\right].
\label{eq:traj2}
\end{equation}
Despite the fact that only the origin is an equilibrium point, its RA is not the whole of $\R^n$ nor is its boundary defined by a limit-cycle. The boundary of the RA is given by $\{x\in \R^n | x_1x_2 = 2 \}$ (obtained from analytical solution to Zubov's equation~\cite[p.73]{Zub64}). 

We fix the shape of the positively invariant sets by fixing $R_M$ as in~\eqref{eq:Rmax}, $d = 2$ with $R_1 = \frac{1}{200}(x_1^2 - 2x_1x_2+ x_2^2),  R_2 = 2x_1x_2$,
and compute $V_N$, $deg(V_N)~=~6$ to define a rational LF $V = \frac{V_N}{\gamma - R_M}$.  We formulate SOS constraints analogous to~\eqref{eq:psatz_thm1} for conditions in Theorem~\ref{thm:invPW}. When solving the constraints we keep $R_M$ constant, i.e.  only the multipliers $m_i$, $p$ and the LF function $V_N$ are updated while increasing $\gamma$.  As a final step, we solve constraints with fixed $\gamma$, $p$, $R_M$ to obtain a rational LF of which the level sets  are depicted in Figure \ref{fig:maxRi}.  

\begin{figure}[!htb]
\begin{psfrags}
     \psfrag{x2}[r][r][1][-90]{\footnotesize $x_2$}
     \psfrag{x1}[l][l]{\footnotesize $x_1$}
\epsfxsize=8cm
\centerline{\epsffile{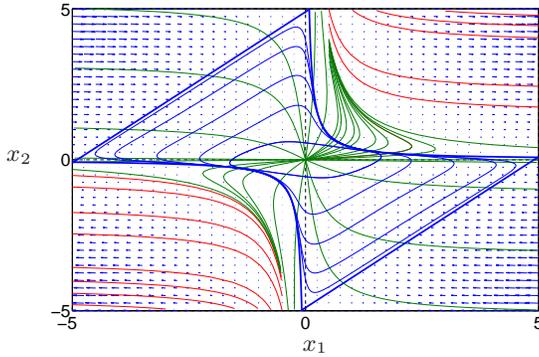}}
\end{psfrags}
\caption{The boundary of the ERA, $\partial\cE(R,\gamma)$, is depicted in dark blue. Trajectories obtained with~\eqref{eq:traj2} are depicted in green (converging) and red (diverging).
\label{fig:maxRi}}
\end{figure}

\section{Conclusion}
\label{sec:conclusion}
In this paper we presented conditions for a positively invariant set to be an ERA of the origin and for Lyapunov certificates given by quotient of two functions where the denominator characterizes the ERA. The main feature of our results is that the positively invariant set defining the ERA  is  \textit{not} necessarily a level set of a Lyapunov function. Provided a stronger condition is satisfied we obtain a Lyapunov function interpretation of the ERA, which for polynomial systems is a rational function. We subsequently proposed an algorithm for the estimation of the RA that guarantees the increment of the estimate at each iteration. 

We then applied the algorithm to the class of polynomial vector fields and semi-algebraic sets for which the steps are performed via the solution to Sum-of-Squares programmes. The extension of the results to the class of systems with parametric uncertainties was also presented. For this case, the positively invariant ERA does not depend on the uncertain parameters while the associated Lyapunov function can.

\bibliographystyle{plain}  
\bibliography{biblio}    
\end{document}